\documentclass[a4paper,11pt]{amsart}
\usepackage{amssymb,amsfonts,amsxtra,
	mathrsfs,graphicx,verbatim,stmaryrd,hyperref,cite,color}
\usepackage[margin=9mm,
marginparwidth=20mm,     
marginparsep=1mm,       
]
{geometry}
\usepackage[all]{xy}
\xyoption{line}
\usepackage{fullpage}
\usepackage{euscript}
\newtheorem{theorem}{Theorem}[section]
\newtheorem{cor}[theorem]{Corollary}
\newtheorem{lem}[theorem]{Lemma}
\newtheorem{prop}[theorem]{Proposition}

\theoremstyle{definition}

\newtheorem{example}[theorem]{Example}
\newtheorem{defi}[theorem]{Definition}
\newtheorem{rem}[theorem]{Remark}

\numberwithin{equation}{section}
\DeclareMathOperator{\SSet}{SSet}
\DeclareMathOperator{\DGA}{DGA}
\DeclareMathOperator{\SGp}{SGp}
\DeclareMathOperator{\Hom}{Hom}

\DeclareMathOperator{\Ker}{Ker}

\DeclareMathOperator*{\Le}{\mathbb L}
\DeclareMathOperator*{\RHom}{\operatorname{RHom}}
\DeclareMathOperator*{\REnd}{\operatorname{REnd}}
\DeclareMathOperator*{\Tor}{\operatorname{Tor}}
\DeclareMathOperator*{\Pe}{\operatorname{P}}

\DeclareMathOperator*{\Perf}{\operatorname{Perf}}
\newcommand{\noproof}{\begin{flushright} \ensuremath{\square}
\end{flushright}}
\def\ground{\mathbf{k}}

\def\id{\operatorname{id}}

\newcommand{\BL}[1]{#1^\wedge}

\thanks{}

\begin{document}
\begin{abstract}
We show that the notions of homotopy epimorphism and homological epimorphism in the category of differential graded algebras are equivalent.
As an application we obtain a characterization of acyclic maps of topological spaces in terms of induced maps of their chain algebras of based loop spaces. In the case of a universal acyclic map we obtain, for a wide class of spaces, an explicit algebraic description  for these induced maps in terms of derived localization.
\end{abstract}
\title[Homological epimorphisms, homotopy epimorphisms and acyclic maps]{Homological epimorphisms, homotopy epimorphisms and acyclic maps}
\author{Joe Chuang, Andrey Lazarev}
\thanks{This work was partially supported by EPSRC grants EP/N015452/1 and EP/N016505/1}

\thanks{}

\maketitle
\tableofcontents
\section{Introduction}
The notion of \emph{epimorphism} exists in any category $\mathcal C$: a morphism $X\to Y$ is an epimorphism
if for any object $Z$ of $\mathcal C$ the induced map of sets $\Hom_{\mathcal C}(Y,Z)\to\Hom_{\mathcal C}(X,Z)$ is injective. Assuming the existence of pushouts in $\mathcal C$, this is equivalent to requiring that the codiagonal map $Y*_XY\to Y$ is an isomorphism cf. for example \cite[Proposition 2.1]{Mur} where this and other easy equivalent reformulations are given. Epimorphisms of sets or groups are just surjections, however already for rings the situation is more interesting e.g. the inclusion $\mathbb Z\to \mathbb Q$ (or any localization of rings) is an epimorphism.
The tensor product $B\otimes_A B$ realises the pushout $B*_A B$ for commutative rings. On the other hand, in the category of (not necessarily commutative) rings, $B\otimes_A B$ is not a pushout, as it is not even a ring in general. However it is still true that $A\to B$ is an epimorphism if and only if the multiplication map $B\otimes_A B\to B$ is an isomorphism \cite[Proposition 4.1.1]{Coh}.

If $\mathcal C$ is a category with an additional structure allowing one to do homotopy theory in it (such as the category of topological spaces or simplicial sets or, more generally, a closed model category) there is a similar notion of a \emph{homotopy epimorphism}: it is a map $X\to Y$ such that $Y*^{\Le}_XY\to Y$ is an isomorphism in the \emph{homotopy category} of $\mathcal C$ where $Y*^{\Le}_XY$ stands for a \emph{homotopy pushout}. It is known \cite{Rap} that homotopy epimorphisms of connected topological spaces  are precisely \emph{acyclic maps}, i.e. maps whose homotopy fibres have zero reduced integral homology groups.

We investigate the notion of homotopy epimorphism in the category of differential graded (dg) algebras, possibly noncommutative. It is known \cite[Proposition 3.17]{BCL} that \emph{derived localizations} of dg algebras are homotopy epimorphisms. On the other hand, derived localizations are also \emph{homological epimorphisms}, i.e. maps $A\to B$ such that
the multiplication map $B\otimes^{\Le}_AB\to B$ is a quasi-isomorphism. The property of being a homological epimorphism has many nice implications for the induced functors on derived categories and it is natural to ask (especially having in mind the corresponding non-homotopy result) whether homotopy and homological epimorphisms are the same thing. Our first main result is that this is, indeed, true.

Next, we consider topological applications. Given a connected space $X$ we denote its based loop space by $GX$; then its chain complex $C_*(GX,\ground)$ (simplicial or singular) with coefficients in a commutative ring $\ground$  is a dg $\ground$-algebra. Our second main result is a characterization of  $\ground$-acyclic maps $f:X\to Y$ as those maps for which the map of dg algebras $C_*(GX,\ground),\to C_*(GY,\ground)$ is a homological epimorphism. Equivalently, $f$ is acyclic if an only if it induces a Verdier localization of the triangulated categories of cohomologically locally constant sheaves (also known as infinity local systems) on $X$ and $Y$ .

Finally, we consider, for a given connected space $X$, its $p$-plus-construction $X^+_p$; then the canonical map $X\to X^+$ is a universal $\mathbb{F}_p$-acyclic map out of $X$ where $\mathbb{F}_p$ is the field with $p$ elements. We show that if $X$ is such that $\pi_1(X)$ has a perfect subgroup of finite index, then the map  $C_*(GX,\mathbb{F}_p)\to C_*(GX^+_p,\mathbb{F}_p)$ admits a purely algebraic description as a certain Bousfield localization of the category of dg $C_*(GX,\mathbb{F}_p)$-modules.  Note that $H_0(GX,\mathbb{F}_p)\cong \mathbb{F}_p[\pi_1(X)]$, the $\mathbb{F}_p$--group ring of $\pi_1(X)$. In the case when $\pi_1(X)$ is \emph{finite} we show that the dg algebra $C_*(GX^+_p,\mathbb{F}_p)$ is (quasi-isomorphic to) the derived localization of $C_*(GX,\ground)$ at a certain idempotent of $\mathbb{F}_p[\pi_1(X)]$. In the case $X$ has no higher homotopy groups, i.e. is the classifying space of a finite group, this is essentially equivalent to the main result \cite{Ben} (which, however, was formulated without invoking derived localizaton).

\subsection{Notation and conventions} We work in the category of $\mathbb Z$-graded (dg) modules over a fixed commutative ground ring $\ground$; an object in this category is a pair $(V,d_V)$ where $V$ is a graded $\ground$-module and $d_V$ is a differential on it; it will always be assumed to be of homological type (so it lowers the degree of a homogeneous element). Unmarked tensor products and homomorphisms will be understood to be taken over $\ground$; we will abbreviate `differential graded' to `dg'. The \emph{suspension} of a graded vector space $V$ is the graded vector space $\Sigma V$ so that $(\Sigma V)_i=V_{i+1}$. Quasi-isomorphisms and isomorphisms will be denoted by $\simeq$ and $\cong$ respectively.

A dg algebra is an associative monoid in the dg category of dg vector spaces with respect to the standard monoidal structure given by the tensor product. Given a map $A\to B$ of dg algebras (not necessarily central) we will refer to $B$ as a dg $A$-algebra. A dg vector space $V$ is a (left) dg module over a dg algebra $A$ if it is supplied with a dg map $A\otimes V\to V$ satisfying the usual conditions of associativity and unitality; a right dg  module is defined similarly.

The categories of dg algebras and dg modules over a dg algebra admit structures of closed model categories; cf. \cite{BCL} for an overview. We will denote by $\DGA_{\ground}$ the category of dg $\ground$-algebras and for a dg algebra $A$ we write $D(A)$ for its derived category. Recall that objects of $D(A)$ are cofibrant left dg $A$-modules and morphisms are chain homotopy classes of dg module maps. For a dg algebra $A$, a left dg  $A$-module $M$ and a right dg $A$-module $N$ there is defined their tensor product $N\otimes_A M$ and their \emph{derived} tensor product $N\otimes^{\Le}_AM$. The latter is defined as either $N^\prime\otimes_A M$ or $N\otimes_A M^\prime$ where $M^\prime$ and $N^\prime$ are cofibrant replacements of $M$ and $N$ respectively. Similarly for two left (or right) dg $A$-modules $M,N$ their derived hom $\RHom_A(M,N)$ is defined as $\Hom(M^\prime,N)$ where $M^\prime$ is a cofibrant replacement of $M$.

The category of reduced simplicial sets will be denoted by $\SSet_*$ and the category of simplicial groups -- by $\SGp$. By `a space' we mean `a simplicial set', nevertheless all the results in the paper are of homotopy invariant nature and so they  make sense and are valid for topological spaces by the well known correspondence between topological spaces and simplicial sets.
\section{Derived free products of dg algebras}
Let $A,B$ and $C$ be graded algebras with $A$ being flat over $\ground$. In this case we can form the free product $B*_A C$, this is again a graded algebra satisfying an appropriate universal property; if $A, B$ and $C$ are dg algebras, then so is $B*_A C$. The derived version $B*^{\Le}_AC$ is described in \cite[Section 2]{BCL}. This is a homotopy pushout in the closed model category of dg algebras and can be defined concretely as $B^\prime *_AC^\prime$ where $B^\prime$ and $C^\prime$ are $A$-cofibrant replacements of $A$-algebras $B$ and $C$ respectively.

For an $A$-algebra $B$ we denote by $\overline{B}$ the cokernel of the unit map $A\to B$; it is clearly a dg $A$-bimodule. Then we have the following important technical result.
\begin{lem}\label{lem:filtration}	
Let $A$ be a dg algebra and $B,C$ be dg $A$-algebras. Assume that the unit maps $A\to B$ and $A\to C$ are cofibrations of left $A$-modules. Then $B*_AC$ has a natural filtration by dg $A$-bimodules $0=F_0\subset B=F_1\subset F_2\subset\ldots$ with  $ \bigcup_n F_n=B*_AC$ and $F_nF_k\subset F_{n+k}$
and such that 
\begin{align}\label{eq:filtration}
F_{2n+1}/F_{2n}\cong B\otimes_A(\overline{C}\otimes_A\overline{B})^{\otimes_{A}n};\\
F_{2n+2}/F_{2n+1}\cong B\otimes_A(\overline{C}\otimes_A\overline{B})^{\otimes_{A}n}\otimes_A\overline{C}
\end{align}	
for $n\geq 0$.
\end{lem}
\begin{proof}
Let $F_{2n}, F_{2n+1}\subset B*_AC$ be spanned by the monomials $b_1c_1\ldots b_nc_n$ and $b_1c_1\ldots b_nc_n b_{n+1}$ with $b_i\in B$ and $c_i\in C$ respectively. It is clear that this filtration consists of dg $A$-bimodules, is exhaustive and multiplicative. Since $A\to B$ and $A\to C$ are cofibrations of left dg $A$-modules, it follows that $\overline{B}$ and $\overline{C}$ 	are cofibrant as left $A$-modules, in particular they are $A$-flat, after forgetting the differential. Now the required conclusion on the associated graded quotients follows from \cite[Theorem 4.6]{Coh1}, cf. also \cite[page 206]{Coh} for a simpler argument in the case when $C$ and $B$ are \emph{free} as left $A$-modules.
\end{proof}	
\begin{rem}\label{rem:weaker}
	In fact, the conclusion of Lemma \ref{lem:filtration} holds under the weaker assumptions that $A\to B$ and $A\to C$ are injections and $\overline{B},\overline{C}$ are flat left $A$-modules since this is what is required for the application of \cite[Theorem 4.6]{Coh1}. Moreover, modifying the filtration so that its components are spanned by monomials \emph{ending} with an element in $B$ rather than beginning with one, one obtains a similar conclusion under the assumption that $\overline{B},\overline{C}$ are flat \emph{right} $A$-modules.
\end{rem}
\begin{cor}\label{cor:derivedfree} Let $A,B, C$ be as in Lemma \ref{lem:filtration}. Then there is quasi-isomorphism of dg algebras $B*_AC\simeq B*^{\Le}_AC$. 	
\end{cor}
\begin{proof}	
	Let $B^\prime, C^\prime$ be $A$-cofibrant replacements of $B$ and $C$ respectively. Then we have a map $$f:B*^{\Le}_AC:=B^\prime *_AC^\prime\to B*_AC.$$
Consider the filtrations on $B*_AC$ and $B^\prime *_AC^\prime$ described in Lemma \ref{lem:filtration}. Then the map $f$ induces a map on the associated graded to these filtrations and, since $\overline{B}, \overline{B^\prime}, \overline{C}, \overline{C^\prime}$ are cofibrant as left $A$-modules, we conclude that the $E_1$-terms of the corresponding spectral sequences are isomorphic and the desired conclusion follows.
\end{proof}
\begin{rem}  It is natural to ask whether the conclusion of Corollary \ref{cor:derivedfree} holds under weaker conditions
	than $A$-cofibrancy of $B$ and $C$, cf. Remark \ref{rem:weaker}. Assuming that the unit maps $A\to B$ and $A\to C$ are injections and that $\overline{B}$ and $\overline{C}$ are flat as left $A$ modules (disregarding the differential) allows to identify the associated graded of the appropriate filtration of  $B*_AC$. However in order to ensure that the iterated tensor product of  $\overline{B}$ and $\overline{C}$ and $B$ computes the \emph{derived} tensor product, one has to assume, in addition, that $\overline{B}$ and $\overline{C}$ are \emph{homotopically flat} as left dg $A$-modules. Recall that a left dg $A$-module $M$ is called homotopically flat if for any right $A$-module $N$ it holds that $N\otimes^{\Le}_A M\simeq N\otimes_A M$; e.g. any cofibrant module is homotopically flat. Thus, $B*_AC$ computes the derived free product if the unit maps $A\to B, A\to C$ are injections and $\overline{B},\overline{C}$ are flat left $A$-modules as well as homotopically flat left $A$-modules (such dg modules are called \emph{semi-flat}, cf. for example \cite{CH} regarding this nomenclature).
	
	Of course, `left' can be replaced with `right' in the above discussion, cf. Remark \ref{rem:weaker}. In particular, the conclusion of Corollary \ref{cor:derivedfree} holds under the assumption that the unit maps $A\to B$ and $A\to C$ are cofibrations of \emph{right} dg $A$-modules.
	
\end{rem}

\section{Modules of relative differentials for dg algebras}
In this section we recall the construction of the  modules of relative differentials for dg algebras. The treatment of \cite{Coh} extends to the dg case in an obvious manner.
\begin{defi}  
For a dg $A$-algebra  $B$ the module of relative differentials	$\Omega_{A}(B)$ is the kernel of the multiplication map:
\begin{equation}\label{eq:redif}
\Omega_A(B)\to B\otimes_AB\to B.
\end{equation}
Thus, $\Omega_A(B)$ is a dg $B$-bimodule. If $A\to B^\prime$ is an $A$-cofibrant replacement of the $A$-algebra $B$, we define the derived module of relative differentials $\Omega_A^{\Le}(B):=\Omega_{A}(B^\prime)$. Thus, $\Omega^{\Le}_A(B)$ is well defined as an object in the homotopy category of dg $B$-bimodules.
\end{defi}	
Suppose that $f:A\to B$ is a cofibration of left dg $A$-modules. The short exact sequence (\ref{eq:redif}) is split as in the category of left dg $B$-modules by a map $B\to B\otimes_AB$; $b\mapsto b\otimes 1$. The cokernel of the latter map is isomorphic as a left dg $B$-module to $B\otimes_A \bar{B}$ (recall that $\bar{B}$ is the cokernel of $f$).  It follows that $B\otimes_A\overline{B}$ can be identified with $\Omega_A(B)$ as a left dg $B$-module. 
	
The formation of the module of relative differentials behaves well with respect to free products:
\begin{lem}\label{lem:differential}
	Let $A,B$ and $C$ be dg algebras. \begin{enumerate}\item There is an isomorphism of dg $B*_AC$-bimodules:
	\[
	\Omega_{A}(B*_AC)\cong \big((B*_AC)\otimes_B\Omega_A(B)\otimes_B (B*_AC)\big)\oplus \big((B*_AC)\otimes_C\Omega_A(C)\otimes_C (B*_AC)\big). 
	\]
	\item 
	The maps $\Omega_A(B)\to \Omega_{A}(B*_AC)$ and $\Omega_A(C)\to \Omega_{A}(B*_AC)$ induced by the canonical maps 
	$B\to B*_AC$ and $C\to B*_AC$ are the compositions
	\[
	\Omega_A(B)\xrightarrow{1\otimes\id\otimes 1} (B*_AC)\otimes_B\Omega_A(B)\otimes_B (B*_AC)\hookrightarrow \Omega_{A}(B*_AC)
	\]
	and
	\[
	\Omega_A(C)\xrightarrow{1\otimes\id\otimes 1}(B*_AC)\otimes_B\Omega_A(C)\otimes_C (B*_AC)\hookrightarrow \Omega_{A}(B*_AC)
	\] \end{enumerate}
\end{lem}
\begin{proof}
Part (1) of the statement above is \cite[Theorem 5.8.8]{Coh}. Unraveling the proof in op. cit. yields (2). 
\end{proof}	

\section{Equivalence of homotopy and homology epimorphisms}
We will now introduce the notions of homotopy and homological epimorphisms for dg algebras and show that they are equivalent.
\begin{defi}
	Let $f:A\to B$ be a map of dg algebras making $B$ into a dg $A$-bimodule.\begin{enumerate}
		\item $f$ is said to be a \emph{homological epimorphism} if the multiplication map $B\otimes^{\Le}_AB\to B$ is a quasi-isomorphism.
		\item $f$ is said to be a \emph{homotopy epimorphism} if the codiagonal map $B*^{\Le}_AB\to B$ is a quasi-isomorphism.
	\end{enumerate}
\end{defi}

\begin{rem}Homotopy epimorphism can be defined in any closed model category. The notion of a homological epimorphism is more restrictive as it requires the structure of a closed model category (or something similar) on modules over monoids in a given closed model category. It would be interesting to investigate whether the equivalence between homological and homotopy epimorphism is a general categorical phenomenon (rather than special to dg algebras).
	
Homological epimorphisms of dg algebras were studied in \cite{Pauk}. Their exceptionally good property is that they induce smashing localization functors on the level of derived categories (indeed they are characterized by this property).
\end{rem}
We have the following characterization of homological epimorphisms, whose proof is a straightforward consequence of definitions.
\begin{lem}\label{lem:acyclicdiff}
	A map $A\to B$ is a homological epimorphism if an only if $\Omega^{\Le}_A(B)$ is acyclic.
\end{lem}
\noproof
\begin{theorem}\label{thm:main}
	A dg algebra map $A\to B$ is a homological epimorphism if and only if it is a homotopy epimorphism. 
\end{theorem}
\begin{proof}
Without loss of generality we can assume that $A\to B$ is a cofibration of dg algebras; then $B*_AB\simeq B*^{\Le}_AB$. Suppose that $A\to B$ is a homological epimorphism. It suffices to show that the map $\id*1:B\cong B*_AA\to B*_AB$ is a quasi-isomorphism. Considering the filtration $\{F_n\}$ on $B*_AB$ constructed in Lemma \ref{lem:filtration} and taking into account that $B\otimes_A \overline{B}\simeq 0$ since $A\to B$ is a homological epimorphism, we conclude that the associated graded quotients $F_n/F_{n-1}$ are acyclic for $n>0$ and so $\id*1:B\to B*_AB$ is indeed a quasi-isomorphism.

Conversely, suppose that $A\to B$ is a homotopy epimorphism; we will show that $\Omega^{\Le}_A(B)\simeq \Omega_A(B)$ is acyclic. We have the following quasi-isomorphisms of $B$-bimodules:
\begin{align*}
\Omega_A(B)&\simeq \Omega_A(B*_AB)\\
&\simeq\big(B*_AB\otimes_B\Omega_A(B)\otimes_B B*_AB\big)\oplus \big(B*_AB\otimes_B\Omega_A(B)\otimes_B B*_AB\big)\\
&\simeq \Omega_A(B)\oplus\Omega_A(B).
\end{align*}	
Here the second quasi-isomorphism follows from Lemma \ref{lem:differential} (1) and the third -- since $B*_AB\simeq B$ and taking into account that $B*_AB$ is cofibrant both as a left and right $B$-module. By Lemma \ref{lem:differential} (2) the map $B\xrightarrow{\id*1} B*_AB$ (that is a quasi-isomorphism) induces a quasi-isomorphism
$\Omega_A(B)\simeq \Omega_A(B)\oplus\Omega_A(B)$ that is also the inclusion into the left direct summand. This is only possible if $\Omega_A(B)\simeq 0$ as required.\end{proof}
\section{$\ground$-acyclic maps}
Recall that there is a left Quillen functor $G:\SSet_*\mapsto\SGp$ that is part of a Quillen equivalence between the categories of reduced simplicial sets and simplicial groups, cf. \cite{GoJ}. We will also consider the left Quillen functor $C_*:\SSet_*\mapsto \DGA_{\ground}$ associating to a simplicial group $H$ its normalized simplicial chain complex supplied with the Pontrjagin product $C_*(H,\ground)\otimes C_*(H,\ground)\to C_*(H,\ground)$ induced by the simplicial group operation $H\times H\to H$. 

Interpreted topologically, the composite functor $C_*\circ G:\SSet_*\to\DGA_{\ground}$ associates to a reduced simplicial set $X$ a dg algebra that is quasi-isomorphic to the chain algebra on the loop space of $|X|$, the geometric realization of $X$. An obvious modification allows one to consider it as a functor from the homotopy category of connected (not necessarily reduced) spaces. 
\begin{defi}
	A map between connected spaces is $\ground$-acyclic if its homotopy fibre  $F$ is $\ground$-acyclic: $H_n(F,\ground)=0, n>0$ and $H_0(F,\ground)=\ground$.
\end{defi}
\begin{rem}
	It is customary to call $\mathbb Z$-acyclic maps simply acyclic.
\end{rem}	
\begin{lem}\label{lem:acyclicmap}
	Let $f:X\to Y$ be a map between connected spaces. Then $f$ is $\ground$-acyclic and induces an isomorphism $\pi_1(X)\cong \pi_1(Y)$ if and only if the induced map of dg algebras $C_*(GX, \ground)\to C_*(GY, \ground)$ is a quasi-isomorphism.
\end{lem}
\begin{proof}
Suppose first that $X$ and $Y$ are both simply-connected. Let $f$ be $\ground$-acyclic. It follows by considering the Serre spectral sequence of $f$ that the dg coalgebras $C_*(X,\ground)$ and $ C_*(Y,\ground)$ are quasi-isomorphic.  Then the spectral sequences associated with cobar-constructions
	of the chain coalgebras $C_*(X,\ground)$ and $ C_*(Y,\ground)$ converge strongly to $H_*(GX, \ground)$ and $H_*(GY, \ground)$ respectively (this is where simple connectivity is needed) and it follows that the dg algebras $C(GX, \ground)$ and $C_*(GY, \ground)$ are indeed quasi-isomorphic.
	
Conversely, if 	the map of dg algebras $C_*(GX, \ground)\to C_*(GY, \ground)$ is a quasi-isomorphism then the bar-construction spectral sequence
implies that $f:X\to Y$ induces an isomorphism on $\ground$-homology and thus, again by simple connectivity and the Serre spectral sequence of $f$, it is a $\ground$-acyclic map. The desired statement is therefore proved in the simply-connected case.

If $X$ and $Y$ are not simply-connected, denote by $\overline{X}$ and $\overline{Y}$ their universal covers and note that the condition that $f:X\to Y$ induces an isomorphism on fundamental groups implies that there is a homotopy  pullback diagram of spaces:
\begin{equation}\label{eq:pullback}
\xymatrix{\overline{X}\ar^{\overline{f}}[r]\ar[d]&\overline{Y}\ar[d]\\
	X\ar^f[r]&Y
}	
\end{equation}
where the map $\overline{f}$ is induced by $f$. Then the homotopy fibre of $\overline{f}$ is $\ground$-acyclic and since $\overline{X},\overline{Y}$ are simply-connected we obtain by the argument above that the dg algebras $C_*(G\overline {X},\ground)$ and $C_*(G\overline {Y},\ground)$ are quasi-isomorphic. But clearly there is a homotopy fibre sequence of simplicial groups
\begin{equation}\label{eq:fibre1}
G\overline{X}\to GX\to GB\pi_1(X)\simeq \pi_1(X)
\end{equation}
and similarly 
\begin{equation}\label{eq:fibre2}
G\overline{Y}\to GX\to GB\pi_1(Y)\simeq \pi_1(Y)
\end{equation}
where $B\pi_1(X), B\pi_1(Y)$ are classifying spaces of $\pi_1(X)$ and $\pi_1(Y)$ respectively. It follows that the maps of dg algebras $C_*(G\overline{X},\ground)\to C_*(GX, \ground)$ and $C_*(G\overline{Y},\ground)\to C_*(GY, \ground)$ induce homology isomorphisms in positive degrees and therefore the map $C_*(GX,\ground)\to C_*(GY, \ground)$ 
also has this property. Finally, $H_0(GX,\ground)\cong \ground[\pi_1(X)]\to \ground[\pi_1(Y)]\cong H_0(GY, \ground)$ is also an isomorphism and so the dg algebras $C_*(GX,\ground)$ and $ C_*(GY, \ground)$ are indeed quasi-isomorphic as claimed.

Conversely, suppose that the induced map $C_*(GX,\ground)\to C_*(GY, \ground)$ is a quasi-isomorphism. In particular it gives an isomorphism $H_0(GX,\ground)\cong\ground[\pi_1(X)]\to \ground[\pi_1(Y)]\cong H_0(GY,\ground)$ which implies that $f$ induces an isomorphism $\pi_1(X)\to \pi_1(Y)$.  Again, we conclude that diagram (\ref{eq:pullback}) is a homotopy pullback. Similarly we conclude from (\ref{eq:fibre1}) and (\ref{eq:fibre2}) that the dg algebras $C_*(G\overline{X},\ground)$ and $C_*(G\overline{Y},\ground)$ are quasi-isomorphic and therefore (because of simple-connectivity of $\overline{X}$ and $\overline{Y})$, the map $\overline{f}$ is a $\ground$-acyclic map. Therefore, by \ref{eq:pullback} $f$ is also a $\ground$-acyclic map.
\end{proof}	
\begin{rem}
	If $\ground=\mathbb Z$ then Lemma \ref{lem:acyclicmap} implies that a map $X\to Y$ between connected spaces is a weak equivalence if and only if the induced map $C_*(GX, \mathbb Z)\to C_*(GY, \mathbb Z)$ is a quasi-isomorphism. Surprisingly, this simple and fundamental fact appears to have been noticed only recently, cf. \cite{CHL,RWZ}.
\end{rem}
\begin{theorem}\label{thm:acyclic}
	Let $f:X\to Y$ be a map between two connected spaces. Then the following are equivalent:
	\begin{enumerate}
		\item $f$ is a $\ground$-acyclic map.
		
		\item The induced map of dg algebras $C_*(GX,\ground)\to C_*(GY,\ground)$ is a homotopy epimorphism.
		
		\item The induced map of dg algebras $C_*(GX,\ground)\to C_*(GY,\ground)$ is a homological epimorphism. 
	\end{enumerate}
\end{theorem}
\begin{proof}
In light of Theorem \ref{thm:main} it suffices to prove the equivalence of (1) and (2). Without loss of generality we assume that $X$ and $Y$ are reduced (as opposed to merely connected) simplicial sets.  The functor $X\mapsto C_*(GX,\ground)$ is a composition of two left Quillen functors and thus, is itself left Quillen and so it preserves homotopy pushouts. It follows that there is a quasi-isomorphism of dg algebras \[C_*(G(Y*^{\Le}_XY),\ground)\simeq C_*(GY,\ground)*^{\Le}_{C_*(GX,\ground)}C_*(GY,\ground).\]
	
Let $f:X\to Y$ be a $\ground$-acyclic map. Then the map $Y\to Y*^{\Le}_XY$ mapping $Y$ to the first (or second) wedge component of  $Y*^{\Le}_XY$ is likewise $\ground$-acyclic and, moreover, by the van Kampen theorem, induces an isomorphism on the fundamental groups. By Lemma \ref{lem:acyclicmap} the corresponding map $C_*(GY,\ground)\to C_*(G(Y*^{\Le}_XY),\ground)\simeq C_*(GY,\ground)*^{\Le}_{C_*(GX,\ground)}C_*(GY,\ground)$ is a quasi-isomorphism, so $C_*(GX,\ground)\to C_*(GY,\ground)$ is a homotopy epimorphism. This chain of implications is clearly reversible and so we obtain the desired if and only if statement.
\end{proof}	
Let $\ground$ be a field. It is known \cite{Hol} that for a connected simplicial set $X$ the derived category $D_{\ground}(|X|)$  of cohomologically locally constant sheaves of $\ground$-modules on $|X|$, the topological realization of $X$ (also known as infinity local systems on $|X|$) is equivalent to $D(C_*(GX,\ground))$, the derived category of the chain algebra on the based loop space of $X$. This leads to the following result.
\begin{cor}
	A map $f:X\to Y$ of connected spaces is $\ground$-acyclic if an only if the inverse image functor $f^{-1}:D_{\ground}(|Y|)\to D_{\ground}(|X|)$ is fully faithful. 
\end{cor}
\begin{proof}
By the correspondence between cohomologically locally constant sheaves and modules on the chain algebra of based loop spaces mentioned above, the  functor $f^{-1}:D_{\ground}(|Y|)\to D_{\ground}(|X|)$ is fully faithful if an only if the restriction functor $D(C_*(GY,\ground))\to D(C_*(GX,\ground))$ is fully faithful. The latter is equivalent by \cite[Theorem 3.9 (6)]{Pauk} to the map $C_*(GY,\ground))\to C_*(GX,\ground)$ being a homological epimorphism which is, by Theorem \ref{thm:acyclic}, is equivalent to $f:X\to Y$ being a $\ground$-acyclic map.
 
\end{proof}	
\section{Plus-construction and derived localization}
\subsection{Recollection on plus-construction, localization and completion of spaces} A $\ground$-plus-construction (called the \emph{partial $\ground$-completion} in \cite{BK}) of a connected space $X$ is a space $X^+_{\ground}$ supplied with a map $X\to X^+_{\ground}$ that is $\ground$-acyclic and is terminal among 
homotopy classes of $\ground$-acyclic maps out of $X$. The space $X^+_{\ground}$ is \emph{local} with respect to $\ground$-acyclic spaces, 
i.e. if $Y$ is $\ground$-acyclic, then any basepointed map $Y\to X^+_{\ground}$ is homotopic to the constant map; and the map $X\to X^+_{\ground}$ is initial among homotopy classes of maps from $X$ into local spaces.

A closely related notion is that of a \emph{localization} with respect to the homology theory $H\ground$; for a space $X$ its  $H\ground$ localization is a space $L_{H\ground}X$ supplied with a map $X\to L_{H\ground}X$ inducing an isomorphism in homology with coefficients in $\ground$ and terminal among homotopy classes of such maps (note that $H\ground$ localization is called \emph{semi-$\ground$-completion} in \cite{BK}). The space $L_{H\ground}X$ is \emph{local} with respect to 
$H\ground$-homology equivalences i.e. if $f:Y\to Z$ is an $H\ground$-homology equivalence then the induced map on homotopy classes 
$[Z,L_{H\ground}X]\to [Y,L_{H\ground}X]$ is an isomorphism. Furthermore, the map $X\to L_{H\ground}X$ is initial among homotopy classes of maps from $X$ into $L_{H\ground}$-local spaces. The space $L_{H\ground}X$ is `farther away' from $X$ than $X^+_{\ground}$ in the sense that the map $X\to L_{H\ground}X$ always factors through $X^+_{\ground}$; in favourable cases $X^+_{\ground}\simeq L_{H\ground}X$.
	
Finally, there is a notion of a $\ground$-completion $X\to X_\ground^{\wedge}$; the latter map always factors through $L_{H\ground}$. Often we have $L_{H\ground}X\simeq X_{\ground}^{\wedge}$, in this case $X$ is called \emph{${\ground}$-good}.

From now on, we will assume that $\ground=\mathbb{F}_p$, although all arguments and statements below are valid for an arbitrary field of characteristic $p$. We will write $X^+_p, X_p$ and $X^{\wedge}_p$ for $X^+_{\mathbb{F}_p}, L_{H\mathbb{F}_p}X$ and $X^{\wedge}_{\mathbb{F}_p}$ respectively.

There is a class of spaces $X$ for which $X^+_p$ is particularly well-behaved. Recall that
a group $G$ is \emph{$p$-perfect} if $H_1(G,\mathbb{F}_p)=0$. The maximal $p$-perfect subgroup of $G$ will be denoted by $\Pe(G)$; it is normal in $G$ and the quotient $G/\Pe(G)$ is $p$-\emph{hypoabelian}, i.e. it does not have nontrivial  $p$-perfect subgroups.
\begin{defi}
	A group $G$ is \emph{$p$-reasonable} if $G$ contains a $p$-perfect subgroup of finite index (equivalently, $G/\Pe(G)$ is a finite $p$-group).
	A connected space $X$ is $p$-reasonable if $\pi_1(X)$ is.
\end{defi}
 \begin{rem}
 Clearly if $G$ is $p$-reasonable then the augmentation ideal in the group ring $\mathbb{F}_p[G/\Pe(G)]$ is nilpotent.
 	Conversely, if $\mathbb{F}_p[G/\Pe(G)]$ has this property then $G/\Pe(G)$ is a finite $p$-group by \cite{Connell} and therefore $G$ is $p$-reasonable.
 \end{rem}
 
\begin{prop}\label{prop:reasonable}
Let $X$ be $p$-reasonable. Then: 
\begin{enumerate}\item $X$ is $p$-good (so that $X^{\wedge}_p\simeq X_p)$.
	\item $\pi_1(X^{\wedge}_p)\cong \pi_1(X)/\Pe(\pi_1(X))$.
	 \item There is a weak equivalence $X_p\simeq X^+_p$.
	 \end{enumerate}	
\end{prop}	
\begin{proof}
Parts (1) and (2) of the required statement are proved \cite[Part III, Proposition 1.11]{AKO} and similar arguments also prove (3). Namely, consider the canonical map $f:X\to B[\pi_1(X)/\Pe(\pi_1(X)]$ induced by the quotient map $\pi_1(X)\to \pi_1(X)/\Pe(\pi_1(X))$ and recall from \cite[Chapter VII, 6.2]{BK} that $X^+_p$ is constructed as a fibrewise completion of the map $f$ (converted into a fibration). Denoting the homotopy fibre of $f$ by $F$ we have therefore a homotopy fibre sequence of spaces
\begin{equation}\label{eq:fibresequence}
F^{\wedge}_p\to X^+_p\to B[\pi_1(X)/\Pe(\pi_1(X)].
\end{equation}
Now apply the functor of $p$-completion to the homotopy fibre sequence $F\to X\to B[\pi_1(X)/\Pe(\pi_1(X))]$. We obtain 
\begin{equation}\label{eq:fibresequence'} F^{\wedge}_p\to X^{\wedge}_p\to B[\pi_1(X)/\Pe(\pi_1(X)]^{\wedge}_p
\end{equation} and this is also a homotopy fibre sequence by the $p$-mod Fibre Lemma of \cite[Chapter II, 5.1]{BK} and taking into account that the $p$-group $\pi_1(X)/\Pe(\pi_1(X))$ acts nilpotently  on $\tilde{H}_*(F,\mathbb{F}_p)$. 
There is a map from (\ref{eq:fibresequence}) to (\ref{eq:fibresequence'}) that is a weak equivalence on end terms (since $B[\pi_1(X)/\Pe(\pi_1(X))]$ is already $p$-complete) and it follows that it is a weak equivalence on the middle terms as required.
\end{proof}
\subsection{Derived localization of dg algebras} Let $A$ be a dg algebra and $s\in H_0(A)$ be a zero-dimensional cycle in $A$. In this situation one can construct \cite{BCL} another dg algebra $L_sA$ together with a dg algebra map $f:A\to L_s(A)$ such that $f(s)$ is invertible in $H_0(L_sA)$ and initial among the homotopy classes of maps out of $A$ with this property.
Moreover, the map $A\to L_sA$ can be interpreted as the Bousfield localization of $A$ as a (left) dg module over itself with respect to the map $r_s:A\to A; r_s(a)=as, a\in A$. It is also the \emph{nullification} of $A$ with respect to $A/s$, the cofibre of $r_s$, cf. \cite[4.10 and Proposition 4.11]{DG} regarding this result and terminology. There is a homotopy fibre sequence of left dg $A$-modules
\begin{equation}\label{eq:cellsequence}
L^sA\to A\to L_sA
\end{equation}
where $L^sA$ is the $s$-colocalization of $A$ (also known as $A/s$-cellularization and denoted by $\operatorname{Cell}_{A/s}(A)$ in \cite{DG}). This is a kind of a dualizing complex for left dg $A$ modules relative to $A/s$ and it has a nice interpretation in terms of classical homological algebra, cf. \cite[Section 4]{DG}:
\[
L^s(A)\simeq {\RHom}_{A}(A/s,A)\otimes^{\Le}_{\REnd_A(A/s,A/s)}A/s.
\]
The following example taken from \cite[Subsection 4.1]{DG} is instructive.
\begin{example}
	Let $A=\mathbb{Z}$, the integers and $s=p$ so $A/s\cong \mathbb{Z}/p$, the cyclic group of prime order $p$. Then $L_p\mathbb{Z}\simeq \mathbb{Z}[\frac{1}{p}]$, and from (\ref{eq:cellsequence}) we obtain $L^p(\mathbb{Z})\simeq\Sigma^{-1}\big(\mathbb{Z}[\frac{1}{p}]/\mathbb{Z}\big)=:\Sigma^{-1}\mathbb{Z}/p^{\infty}$.
\end{example}
Assume that $s=e$ is an idempotent in $H_0(A)$; then $L_eA$ is quasi-isomorphic as a left $A$-module to the Bousfield localization of $A$ with respect to the localizing subcategory generated by the left dg $A$-module $A(1-e)$ (since $A(1-e)$ and $A/e\simeq A(1-e)\oplus \Sigma A(1-e)$ generate the same localizing subcategory). Therefore, by \cite{DG} there is a  quasi-isomorphism of left dg $A$-modules:
\begin{align}\label{eq:cellularization}\begin{split}
L^e(A)\simeq& {\RHom}_{A}(A(1-e),A)\otimes^{\Le}_{\REnd_A(A(1-e),A(1-e))}A(1-e)\\ \simeq& A(1-e)\otimes^{\Le}_{(1-e)A(1-e)}(1-e)A.
\end{split}
\end{align}
The $A$-bimodule $I:=A(1-e)\otimes^{\Le}_{(1-e)A(1-e)}(1-e)A$ can be viewed as a `derived two-sided ideal' generated by the idempotent $1-e$ in $A$ in the sense that the homotopy cofibre $A/I$ (`derived quotient') of $A$ by $I$ is quasi-isomorphic to the derived localization $L_eA$. This resembles quotienting out by a two-sided ideal generated by an idempotent in a nonderived context.
\subsection{Algebraic description of the loop space of a $p$-plus-construction}Let $X$ be a connected space; below we will write $GX^+_p$ and $GX^{\wedge}_p$ for $G(X^+_p)$ and $G(X^{\wedge}_p)$ respectively. It follows from Theorem \ref{thm:acyclic} that the map of dg algebras $C_*(GX, \mathbb{F}_p)\to C_*(GX^+_p,\mathbb{F}_p)$ is a homology epimorphism and so, it is natural to ask for an algebraic description of this map.  We will give such a description for a $p$-reasonable space; the result is particularly pleasant when $\pi_1(X)$ is a finite group.

Set $A:=C_*(GX,\mathbb{F}_p)$. There is a canonical map $C_*(GX,\mathbb{F}_p)\to H_0(GX,\mathbb{F}_p)\cong \mathbb{F}_p[\pi_1(X)]$ in the homotopy category of dg $\mathbb{F}_p$-algebras, and, since $\mathbb{F}_p[\pi_1(X)]$ is augmented, so is $C_*(GX,\mathbb{F}_p)$.  Consider the homology functor $H\mathbb{F}_p:M\mapsto \mathbb{F}_p\otimes_A^{\Le}M$ 
on the category of dg $A$-modules. The notation $H\mathbb{F}_p$ is designed to invoke an analogy with the homology functor in the stable homotopy category given by smashing with the mod-$p$ Eilenberg MacLane spectrum (in fact, this is not merely an analogy since the category of dg algebras is Quillen equivalent to the category of algebras over the integral Eilenberg-MacLane spectrum $H\mathbb{Z}$, cf. \cite{Shipley}).   We will denote by $\BL{M}$ the Bousfield localization of an $A$-module $M$ with respect 
to this homology functor. Thus, we have a canonical map $M\to \BL{M}$ that is an 
$H\mathbb{F}_p$-equivalence (i.e. induces a quasi-isomorphism upon applying $H{\mathbb{F}_p}$) and such that $\BL{M}$ is $H\mathbb{F}_p$-local (i.e. it does not admit homotopy nontrivial maps from $H\mathbb{F}_p$-acyclic modules).
\begin{theorem}\label{thm:localizationkernel} Let $X$ be a connected space. Then:\begin{enumerate}
\item If $X$ is $p$-reasonable then $C_*(GX^+_p,\mathbb{F}_p)$ is quasi-isomorphic to 
$\BL{C_*(GX,\mathbb{F}_p)}$ and the canonical map 
$C_*(GX,\mathbb{F}_p)\to C(GX^+_p,\mathbb{F}_p)$ 
is the $H\mathbb{F}_p$ localization map on the category of left dg $C_*(GX,\mathbb{F}_p)$-modules.
\item If $\pi_1(X)$ is a finite group then there exists an idempotent $e$ in $\mathbb{F}_p[\pi_1(X)]$ such that the dg algebra $C_*(GX^+_p,\mathbb{F}_p)$ is quasi-isomorphic to the derived localization $L_eC_*(GX,\mathbb{F}_p)$ and $C_*(GX,\mathbb{F}_p)\to C_*(GX^+_p,\mathbb{F}_p)$ is the $e$-localization map. 
\end{enumerate}
\end{theorem}
\begin{proof}
	Let $X$ be $p$-reasonable, $A:=C_*(GX,\mathbb{F}_p)$ and $LA:=C_*(GX^+_p,\mathbb{F}_p)$. To show that $LA\simeq \BL{A}$ we need to show that \begin{enumerate}
		\item[(a)] 
	the map $A\to LA$ is an $H\mathbb{F}_p$-equivalence and\item[(b)] 
	 that $LA$ is $H\mathbb{F}_p$-local.
	\end{enumerate} 
Denote by $M$ the homotopy fibre of $A\to LA$; then $M\otimes^{\Le}_A LA\simeq 0$ and so
\[
\mathbb{F}_p\otimes^{\Le}_AM\simeq \mathbb{F}_p\otimes^{\Le}_ALA\otimes_{LA}^{\Le}M\simeq 0,
\]	
which implies (a). Next, note that any dg $A$-module whose $A$-action is through the augmentation map $A\to \mathbb{F}_p$, is $H\mathbb{F}_p$-local. Indeed if $Y$ is such an $A$-module and $N$ is  $H\mathbb{F}_p$-acyclic (i.e. $\mathbb{F}_p\otimes^{\Le}_AN\simeq 0$) then
\[
{\RHom}_{A}(N,Y)\simeq{\RHom}_{\mathbb{F}_p}(\mathbb{F}_p\otimes^{\Le}_AN,Y)\simeq 0.
\]	

Since $X$ is $p$-reasonable, $H_0(LA)\cong\mathbb{F}_p[\pi_1(X^+_p)]$ is the $\mathbb{F}_p$-group ring of a finite $p$-group and is, thus, a local $\mathbb{F}_p$-algebra with a finite-dimensional nilpotent maximal ideal. Any dg $H_0(LA)$-module has a finite filtration induced by the powers of the maximal ideal in $H_0(LA)$ with associated graded quotients being $\mathbb{F}_p$-modules, which implies that any such module is $H\mathbb{F}_p$-local. The dg $A$-module $LA$ can be represented as a homotopy inverse limit of its Postnikov tower $\{LA[n],n=0,1,\ldots\}$ (so that $LA[n]$ has the same homology as $LA$ up to and including degree $n$ and zero in higher degrees). The homotopy cofibres $LA[n]/LA[n+1]$ are isomorphic as objects in $D(A)$ to dg $H_0(LA)$-modules and thus, are $H\mathbb{F}_p$-local. This implies that $LA$ is $H\mathbb{F}_p$-local and part (1) is therefore proved.

For part (2)  let $e$ be a primitive idempotent
of $\mathbb{F}_p[\pi_1(X)]$ which acts as the identity on the trivial module of $\mathbb{F}_p[\pi_1(X)]$ and by 0 on other
simple modules; it is unique up to conjugation by a unit of $\mathbb{F}_p[\pi_1(X)]$. Then the localization map $f:A\to L_eA$ is easily seen to be the $H\mathbb{F}_p$-localization so the statement follows by part (1). Indeed, the homotopy fibre $A^e$ of $f$ has the property that $L_eA^e\simeq 0$ but then \begin{align*}\mathbb{F}_p\otimes^{\Le}_AA^e&\simeq \mathbb{F}_p\otimes^{\Le}_AL_eA\otimes_{A}A^e\\
&\simeq \mathbb{F}_p\otimes_A^{\Le} L_eA^e\simeq 0\end{align*} so $A^e$ is $H\mathbb{F}_p$-acyclic and $f$ is an $H\mathbb{F}_p$-local equivalence. Also $H_0(L_eA)$ is a local ring with residue field $\mathbb{F}_p$ so $L_eA$ is $H\mathbb{F}_p$-local.
\end{proof}
The following is a straightforward consequence of Theorem \ref{thm:localizationkernel} (1).
\begin{cor}\label{cor:Fploc}
Let $X$ be a $p$-reasonable space.	The homotopy epimorphism of dg algebras $C_*(GX,\mathbb{F}_p)\to C_*(GX^{\wedge}_p,\mathbb{F}_p)$ induces a smashing localization functor on the derived categories \[D(C_*(GX,\mathbb{F}_p))\to D(C_*(GX^{\wedge}_p,\mathbb{F}_p))\] that coincides with $H\mathbb{F}_p$-localization on the  \emph{perfect} subcategory of $ D(C_*(GX^{\wedge}_p,\mathbb{F}_p))$
\end{cor}
\noproof
\begin{rem}
	Note that the  $H\mathbb{F}_p$-localization functor on the full derived category of $C_*(GX,\mathbb{F}_p)$ is not necessarily smashing, even if it is so on the perfect subcategory.
\end{rem}	
\begin{cor}\label{cor:benson}
	Let $X$ be a connected space with $\pi_1(X)$ finite and denote by $e$ an idempotent in $\mathbb{F}_p[\pi_1(X)]$ acting as the identity on the trivial $\pi_1(X)$-module and zero on other simple $\pi_1(X)$-modules. Then there is a homotopy cofibre sequence of dg $C_*(GX, \mathbb{F}_p)$-modules
	\[
	C_*(GX, \mathbb{F}_p)(1-e)\otimes^{\Le}_{(1-e)C_*(GX, \mathbb{F}_p)(1-e)}(1-e)C_*(GX, \mathbb{F}_p)\to C_*(GX, \mathbb{F}_p)\to C_*(GX^+_p, \mathbb{F}_p)
	\]
\end{cor}
\begin{proof}
This follows directly from Theorem \ref{thm:localizationkernel} (2), taking into account (\ref{eq:cellsequence}) and (\ref{eq:cellularization}).
\end{proof}	
\begin{rem}
Let us consider the case $X=BH$ where $H$ is a finite group; the homotopy theory of the loop space on $BH^{\wedge}_p\simeq BH^+_p$ has been studied extensively, see the review paper \cite{CoL} and the more recent \cite{Ben}. From Corollary \ref{cor:benson} we obtain a homotopy fibre sequence
\begin{equation}\label{eq:torsequence}
\mathbb{F}_p[H](1-e)\otimes^{\Le}_{(1-e)\mathbb{F}_p[H](1-e)}(1-e)\mathbb{F}_p[H]\to \mathbb{F}_p[H]\to C_*(G(BH^+_p), \mathbb{F}_p).
\end{equation}
Next, taking the homology long exact sequence of (\ref{eq:torsequence}) we obtain, for $n>1$:
\[
H_n(G(BH^+_p),\mathbb{F}_p)\cong {\Tor}_{n-1}^{(1-e)\mathbb{F}_p[H](1-e)}\big(\mathbb{F}_p[H](1-e),(1-e)\mathbb{F}_p[H]\big)
\]
while for $n=1$ we have an exact sequence
\[
0\to H_1(G(BH^+_p),\mathbb{F}_p)\to 
\mathbb{F}_p[H](1-e)\otimes_{(1-e)\mathbb{F}_p[H](1-e)}(1-e)\mathbb{F}_p[H]\to 
\mathbb{F}_p[H]\to 
\mathbb{F}_p[H/\Pe(H)]\to 0
\]
which is the main result of \cite{Ben} obtained by different methods.
\end{rem}

Since the localization functor $L:D(C_*(GX, \mathbb{F}_p))\to D(C_*(GX^+_p, \mathbb{F}_p))$ is always smashing, it makes sense to ask whether it is \emph{finite}, i.e. whether the localizing subcategory $\Ker(L)\subseteq D(C_*(GX, \mathbb{F}_p))$  is compactly generated. Theorem \ref{thm:localizationkernel} (2) tells us that it is the case when $\pi_1(X)$ is finite (indeed in this case it is even a derived localization at a single element). It turns out that $\Ker L$ may not be compactly generated, even when $X$ is a classifying space of an abelian group. To see that, let us note first the following useful derived analogue of Nakayama's lemma; here $\Perf(A)$ stands for the perfect derived category of a ring $A$.
\begin{lem}\label{lem:nakayama} Let $A$ be an ordinary ring and let $S=A/J$ be the quotient by its Jacobson radical $J$. Then	the functor $S\otimes^{\Le}_A-:\Perf(A)\to\Perf(S)$ has the zero kernel.
\end{lem}
\begin{proof}
This is \cite[Lemma 5.3]{Kra}.
\end{proof}
	
\begin{example}
Let $G$ be a locally finite $p$-group; then  $\mathbb{F}_p[G]$ is local by \cite{Nik} and so, its augmentation ideal coincides with the Jacobson radical. Assume additionally, that $G$ is divisible; e.g.  $G=\mathbb{Z}/p^{\infty}\cong\bigcup_{n=1}^\infty \mathbb{Z}/p^{n}$.  It follows from divisibility of $G$ that $H_1(G,\mathbb{F}_p)=0$ and so, $G$ is $p$-perfect. So, by Theorem \ref{thm:localizationkernel} (1) we have the smashing localization functor $L:D(\mathbb{F}_p[G])\to D(L_{H\mathbb{F}_p}\mathbb{F}_p[G])$; moreover $L$ coincides with the $H\mathbb{F}_p$-localization when restricted to perfect dg $\mathbb{F}_p[G]$-modules by Corollary \ref{cor:Fploc}. Thus, a perfect dg $\mathbb{F}_p[G]$-module $M$ is in $\Ker L$ if and only if  $H_*(\mathbb{F}_p\otimes^{\Le}_{ \mathbb{F}_p[G]}M)=0$ and by Lemma \ref{lem:nakayama}, $M\simeq 0$.
\end{example}
\begin{rem}
	The above example is a modification of Keller's counterexample to the Telescope Conjecture \cite{Kel}. 
\end{rem}
\begin{example}
Let $p,q$ be prime numbers such that $q$ divides $p-1$. Then $\mathbb{Z}/q$ acts faithfully on $\mathbb{Z}/p$ and we can form the semidirect product
$\mathbb{Z}/p\rtimes\mathbb{Z}/q$, clearly it is $p$-perfect. Let $X:=B(\mathbb{Z}/p\rtimes\mathbb{Z}/q)$ An easy calculation with the Serre-Hochschild spectral sequence associated with the normal subgroup $\mathbb{Z}/p$ of $\mathbb{Z}/p\rtimes\mathbb{Z}/q$  shows that there is an isomorphism of graded algebras \[H^*(X^+_p,\mathbb{F}_p)\cong H^*(\mathbb{Z}/p\rtimes\mathbb{Z}/q, \mathbb{F}_p)\cong \mathbb{F}_p[x^\prime, y^\prime]\] with $x^\prime,y^\prime$ situated in cohomological degrees $2q-1$ and $2q$ respectively. Since $X^+_p$ is simply-connected, this implies that its lowest homotopy group is $\pi_{2q-1}X^+_p\cong \mathbb{Z}/p$. It follows that there exists a homotopy fibre sequence
\[
(S^{2q-1})^{\wedge}_p\xrightarrow{j} (S^{2q-1})^{\wedge}_p\rightarrow X^+_p
\]
where $j$ is the self-map of the $p$-completed $2q-1$-sphere $(S^{2q-1})^{\wedge}_p$ of degree $2q-1$ (this argument is presented in \cite[Chapter VII, Proposition 4.4]{BK} for $p=3,q=2$). Thus, $GX^+_p$ is the homotopy fibre of $j$, in particular the latter is a loop space. This is a well-known fact and it is also known that this loop space structure is unique, \cite[Corollary 1.2]{BL}.

Now let $p>3$ or $q>2$. The homology of $GX^+_p$ can be computed with the help of the cobar spectral sequence whose $E_2$ term is
$\operatorname{Ext}_{H^*(X^+_p,\mathbb{F}_p)}(\mathbb{F}_p,\mathbb{F}_p)\cong \mathbb{F}_p[x,y]$ with $|x|=2q-2$ and $|y|=2q-1$.  Taking into account the Hopf algebra structure on this spectral sequence with $x,y$ primitive, we conclude that it collapses and there is an isomorphism of graded algebras 
$H_*(GX^+_p,\mathbb{F}_p)\cong \mathbb{F}_p[x,y]$.  Note that the dg algebra $C_*(GX^+_p,\mathbb{F}_p)$ is \emph{not} formal 
i.e. it is not quasi-isomorphic to its homology. Indeed, if it were, then $C^*(X^+_p,\mathbb{F}_p)$ would likewise be formal and quasi-isomorphic to $H^*(X^+_p,\mathbb{F}_p)\cong\mathbb{F}_p[x^\prime, y^\prime]$, however such a result is incompatible with the minimal model of $C^*(\mathbb{Z}/p,\mathbb{F}_p)$ computed in \cite{Mad}.
	
The case $p=3,q=3$ requires a separate treatment since the Hopf algebra structure on the cobar spectral sequence for $GX^+_p$ is insufficient for deducing the multiplicative structure on $H_*(GX^+_p,\mathbb{F}_p)$. Note also that in this case $X\cong BS_3$, the classifying space of the symmetric group on three letters; an idempotent in $\mathbb{F}_3[S_3]$ figuring in the formulation of Theorem \ref{thm:localizationkernel} can be found explicitly; some possible choices are $-(12)-1, -(13)-1$ and $-(23)-1$. Let $e$ be any one of these.  Then the dg algebra $L_e\mathbb{F}_3[S_3]\simeq C_*(GX^+_3,\mathbb{F}_3)$ is formal and quasi-isomorphic to the graded  algebra (with zero differential) generated over $\mathbb{F}_3$ by two indeterminates $x,y$ with $|x|=2$ and $|y|=3$ subject to the relation $xy=yx$ and $x^3=y^2$ (note that this is \emph{not} a graded commutative algebra, although its associated graded with respect to the powers of the maximal ideal $(x,y)$ is isomorphic to the graded commutative polynomial algebra $\mathbb{F}_3[x,y]$).

To obtain this result, we will realise $\BL{\mathbb{F}_3[S_3]}$ as a `squeezed resolution', following Benson \cite{Ben}.
Let $\alpha:Ae \to A(1-e)$ and $\beta:A(1-e)\to Ae$  be nonzero homomorphisms; both are unique up to multiplication by a nonzero scalar. Then the dg $A$-module $P=(P_i,d_i)$ given by
$$
P_i = \begin{cases} A(1-e) & \text{if } i>0 \\ Ae & \text{if } i=0  \\ 0 & \text{if } i<0 \end{cases},
\qquad
d_i = \begin{cases} \alpha\circ\beta & \text{if } i>1 \\ \beta & \text{if } i=1  \\ 0 & \text{if } i<1 \end{cases}
$$
is a left $\mathbb{F}_3$-squeezed resolution for $S_3$, in the sense of \cite[Definition 3.2]{Ben}. The chain maps $z(n):P\to\Sigma^{-n} P$ 
given by
$$
z(n)_i = \begin{cases} \Sigma^{-n} & \text{if } i>0 \\ \Sigma^{-n}\circ\alpha & \text{if } i=0  \\ 0 & \text{if } i<n \end{cases}
$$
for $n\geq 2$, together with the identity map on $P$, descend to a basis of $\Hom^\ast_{D(A)}(P,P)\cong \Hom^\ast_{D(A)}(A,P)\cong H(P)$.
Since 
 $z(m)\circ z(n)= z(m+n)$ for all $m,n\geq 2$, as is easily verified, it follows that
 $L_e\mathbb{F}_3[S_3]\simeq \RHom^{\ast}_A(P,P)$ is quasi-isomorphic to the subalgebra generated by $z(2)$ and $z(3)$, which is a graded algebra with the claimed presentation.
\begin{rem}
	Some examples of computation of $H_*(GX^+_p,\mathbb{F}_p)$ for $X$ a classifying space of a finite group are presented in \cite[Section 13C]{BGS}, in particular
	when this finite group is $\mathbb{Z}/p\rtimes \mathbb{Z}/q$. The special case $p=3,q=2$ was not noticed in op. cit.
\end{rem} 
 
\end{example}


\begin{thebibliography}{99}
\bibitem{AKO} M. Aschbacher, R. Kessar and R. Oliver,  {\em Fusion systems in algebra and topology}, London Mathematical Society Lecture Note Series, 391. Cambridge University Press, Cambridge, 2011.	
%
\bibitem{Ben} D. Benson,  {\em
	An algebraic model for chains on $\Omega BG^{\wedge}_p$}. 
Trans. Amer. Math. Soc. 361 (2009), no. 4, 2225–-2242. 
%
\bibitem{BL}
C. Broto and R. Levi,
{\em Loop structures on homotopy fibers of self maps of spheres},
Amer. J. Math. 122 (2000), no. 3, 547–-580.	
%
\bibitem{BGS} D. Benson, J. P. C. Greenlees and S. Shamir, {\em Complete intersections and mod p cochains}, Algebr. Geom. Topol. 13 (2013), no. 1, 61-–114.
%
\bibitem{BK}
A. K. Bousfield and D. Kan, {\em Homotopy limits, completions and localizations}, Springer-Verlag, Berlin, 1972. Lecture Notes in Mathematics, Vol. 304.
%
\bibitem{BCL} C. Braun,  J. Chuang and A. Lazarev, {\em Derived localisation of algebras and modules.} Adv. Math. 328 (2018), 555--622.
%
\bibitem{CoL} F. Cohen and R. Levi, {\em On the homotopy theory of p-completed classifying spaces}. Group representations: cohomology, group actions and topology (Seattle, WA, 1996), 157–-182, Proc. Sympos. Pure Math., 63, Amer. Math. Soc., Providence, RI, 1998.
%
\bibitem{CH}  L. W. Christensen and H. Holm, {\em The direct limit closure of perfect complexes}, Journ. of Pure and Appl. Algebra 219, no. 3, pp. 449--463, 2015.
%
\bibitem{Coh1} P. M. Cohn, {\em On the free product of associative rings.} Math. Z. 71, 1959, 380-–398.		
%
\bibitem{Coh} P.M. Cohn, {\em Skew fields. Theory of general division rings.} Encyclopedia of Mathematics and its Applications, 57. Cambridge University Press, Cambridge, 1995. 
%
\bibitem{Connell} I. G. Connell,  {\em On the group ring},
Canadian J. Math. 15 (1963), 650–-685.
%
\bibitem{DG} W. G. Dwyer and J. P. C. Greenlees,
{\em Complete modules and torsion modules},
Amer. J.Math. 124 (2002), no. 1, 199–-220.
%
\bibitem{GoJ} P. G. Goerss and J. F.  Jardine,  {\em Simplicial homotopy theory.} Progress in Mathematics, 174, Birkh\"auser Verlag, Basel, 1999.
%
%
\bibitem{Hol} J. V. S. Holstein,  {\em Morita cohomology}. Math. Proc. Cambridge Philos. Soc. 158 (2015), no. 1, 1--26.
%
\bibitem{CHL}  J. Chuang, J. Holstein and A. Lazarev, {\em Homotopy theory of monoids and derived localization.}, \texttt{arXiv:1810.00373 }.
%
\bibitem{Kel} B. Keller, {\em A remark on the generalized smashing conjecture}, Manuscripta Math. 84 (1994), 193--198.
%
\bibitem{Kra} H. Krause, {\em
Cohomological quotients and smashing localizations},
Amer. J. Math. 127 (2005), no. 6, 1191-–1246.
%
\bibitem{Mad} D. Madsen {\em Homological aspects in representation theory}, Norwegian Univ. Sci. Technol., Trondheim, 2002
%
\bibitem{Mur} F. Muro, {\em Homotopy units in A-infinity algebras.} Trans. Amer. Math. Soc. 368 (2016), no. 3, 2145--2184.
%
\bibitem{Nik} W. K. Nicholson,  \emph{Local group rings},
Canad. Math. Bull. 15 (1972), 137–-138.
%
\bibitem{Pauk} D. Pauksztello, {\em Homological epimorphisms of differential graded algebras.} Comm. Algebra 37 (2009), no. 7, 2337--2350.
%
\bibitem{Rap} G. Raptis,  {\em Some characterization of acyclic maps}. Homotopy Relat. Struct. (2019) vol. 14, issue 3,  773--785.
%
\bibitem{RWZ} M. Rivera, F. Wierstra, and M. Zeinalian, The functor of singular chains detects weak
homotopy equivalences, Proc. AMS, 147 (2019), 4987--4998.
\bibitem{Shipley} B. Shipley, \emph{
$H\mathbb{Z}$--algebra spectra are differential graded algebra}, 
Amer. J. Math. 129 (2007), no. 2, 351--379. 
    \end{thebibliography}
\end{document}